\documentclass[11pt]{article}
\usepackage{mymacros}

\newcommand{\dsspaper}{}
\newcommand{\DSS}{$\mathrm{DSS}$}

\author{
  Omar Abdelghani\footnote{Courant Institute of Mathematical Sciences, New York University. E-mail: {\tt oa2391@nyu.edu}.}
  \and Roland Bauerschmidt\footnote{Courant Institute of Mathematical Sciences, New York University. E-mail: {\tt bauerschmidt@cims.nyu.edu}.}
  \footnote{Current address: Institut f\"ur Angewandte Mathematik, Universit\"at Bonn. E-mail: {\tt bauerschmidt@uni-bonn.de}.}
  \and Thierry Bodineau\footnote{I.H.E.S., Universit\'e Paris-Saclay, CNRS, Laboratoire Alexandre Grothendieck. 
 E-mail: {\tt bodineau@ihes.fr}.}
  \and Benoit Dagallier\footnote{Department of Mathematics, Imperial College London. E-mail: {\tt b.dagallier@imperial.ac.uk}.}
}

\title{The Ding-Song-Sun inequality for a class of even ferromagnets}
\date{September 22, 2026}

\begin{document}

\setcounter{page}{0}
\newpage

\maketitle
\begin{abstract}
Ding-Song-Sun (DSS) proved the remarkable correlation inequality that the truncated two-point correlation
function of ferromagnetic Ising models with arbitrary mixed-sign external field is largest when the
field vanishes. This inequality has various important consequences such as log-Sobolev inequalities
for Ising and $\varphi^4$ models up to and at the critical point.
Using the maximum principle for a version of the Polchinski flow equation,
we extend the DSS inequality to the class of single spin measures satisfying the GHS condition.
The method also extends to various other correlation inequalities such as the FKG and GHS inequalities.
\end{abstract}

\begin{center}
  \emph{Dedicated to Chuck Newman on the occasion of his 80th birthday.}
\end{center}

\section{Ding--Song--Sun inequality for potentials in the GHS class}
\label{sec:DSS}


The goal is to prove that,
for all finite sets $\Lambda$, all ferromagnetic matrices $A\in\R^{\Lambda\times\Lambda}$, all
even potentials $V=V_x:\R\to\R$ ($x\in\Lambda$) satisfying the 
Griffiths--Hurst--Sherman (GHS) condition
\begin{equation}
\forall (g,h)\in\R\times\R_+,\qquad 
V'(g+h)-V'(g-h) \geq V'(h)-V'(-h)
,
\label{eq_prop_GHS}
\end{equation}
and which are unbounded above,
one has the Ding--Song--Sun (DSS) inequality:
\begin{equation}
\forall (g,h)\in\R^\Lambda \times \R_+^\Lambda,
\forall x\in\Lambda,\qquad 
\langle \sigma_x\rangle_{g+h}-\langle \sigma_x\rangle_{g-h} 
\leq 
\langle \sigma_x\rangle_{h}-\langle \sigma_x\rangle_{-h} 
.
\label{eq_prop_DSS}
\end{equation}
The expectation $\avg{\cdot}_u$ where $u\in\R^\Lambda$ is with respect to the following probability measure on $\R^\Lambda$:
\begin{equation}
\mu_{A,V}^{u}(d\sigma)
\propto
\exp\bigg[-\frac{1}{2}(\sigma,A\sigma) + (\sigma,u) - \sum_{x\in\Lambda}V_x(\sigma_x)\bigg]
\, d\sigma
,\qquad (u,v) = \sum_{x\in\Lambda} u_xv_x
.
\end{equation}
%
%
%
The matrix $A$ is ferromagnetic if $A_{xy} \leq 0$ for $x \neq y$.

\begin{theorem}\label{thm:main}
  Assume that $A$ is ferromagnetic, strictly positive definite, and that the $V=V_x \in C^3$ are symmetric, unbounded above, with $V'$ convex on $[0,\infty)$.
  Then the DSS inequality \eqref{eq_prop_DSS} holds.
\end{theorem}

\begin{remark}\label{rk:GHS}
The condition~\eqref{eq_prop_GHS} is equivalent to that of the GHS class of even, unbounded
potentials $V:\R\to\R$ for which $V'$ is convex on $[0,\infty)$, introduced in \cite{MR395659}. 
Indeed, by symmetry, \eqref{eq_prop_GHS} is equivalent to
\begin{equation}
\forall (g,h)\in\R\times\R_+,\qquad 
V'(g+h)+V'(h - g) \geq 2 V'(h).
\label{eq_prop_GHS symmetric}
\end{equation}
This implies that $V''(h) \geq 0$ by letting $g$ tend to 0. For the reverse implication,
note that, if $g \in [0,h]$, then \eqref{eq_prop_GHS symmetric} is a consequence of the convexity of $V'$ on $[0,\infty)$.
For $g \geq h$, \eqref{eq_prop_GHS} can be rewritten as 
\begin{equation}
V'(g+h) - V'(g) + V'(g)  -V'(g-h) \geq 2 (V'(h)-V'(0)),
\end{equation}
which also follows from the convexity of $V'$ on $[0,\infty)$.
\end{remark}

\ifdefined\dsspaper
In \cite{LSIShG}, we apply our DSS inequality to show that the massless sinh-Gordon model in $\R^2$
satisfies a uniform log-Sobolev inequality, by applying the method of \cite{MR4720217,MR4303014,MR4798104}.
\fi

For Ising models (and by approximation $\varphi^4$ models), the DSS inequality was first proved in a remarkable article
by Ding, Song, and Sun \cite{MR4586225}.
The GHS class was identified by Ellis, Monroe, and Newman as the necessary and sufficient
condition for the concavity of the magnetisation for all ferromagnetic couplings $A$ in \cite{MR395659}.
The concavity of the magnetisation implies that the truncated two-point function is maximal at $0$ for constant external fields.
The DSS inequality implies the more general statement that this holds for arbitrary mixed-sign external fields \cite{MR4586225}:
\begin{equation}
\label{eq: 2 point DSS estimate}
  \forall g\in\R^\Lambda,
  \forall x,y\in \Lambda, \qquad
  \langle \sigma_x;\sigma_y\rangle_g
  \leq \langle \sigma_x;\sigma_y\rangle_0,
\end{equation}
where $\langle \cdot;\cdot\rangle_u$ denotes the covariance with respect to the probability measure $\mu_{A,V}^u$.

Our proof of Theorem~\ref{thm:main} relies on the parabolic maximum principle applied to a variant of the Polchinski (renormalisation group) equation;
see \cite{MR4798104} for an introduction.
The idea is to derive a parabolic partial differential equation for a function $W_t$
so that the DSS inequality \eqref{eq_prop_DSS} is encoded as $W_1 \geq 0$  
and the assumption \eqref{eq_prop_GHS} in the initial condition as $W_0 \geq 0$.
This equation for $W$ is related to the Polchinski equation for the doubled system.
The DSS inequality then follows from the maximum principle which implies $W_t \geq 0$ for all $t$.
The same strategy works (more easily) for other correlation inequalities such as the FKG and GHS inequalities
(without duplication of variables as in the standard proof of the GHS inequality);
these applications will be presented elsewhere.

\subsection{Regularisation}
\label{sec: Regularisation}

The following lemma shows that
it suffices to prove the DSS inequality under the additional regularity assumption that $V''$ is bounded.

\begin{lemma} \label{lem:reg}
Let $V$ be as in Theorem~\ref{thm:main}. 
Then,
for any $R>0$ large enough, there is an even function $V_R \in C^3$ that is unbounded above, with $V'_R$ convex on $[0,\infty)$ such that the following holds: 
$V_R''$ is bounded, constant and strictly positive at $\infty$,  
and for any fixed external field $g \in \R^\Lambda$, 
\begin{equation}
   \lim_{R\to\infty}\E_{\mu_{A,V_R}^g}[\sigma_x] 
   =
   \E_{\mu_{A,V}^g}[\sigma_x]
   ,\qquad 
   x\in\Lambda
   .
   \label{eq_limit_VR}
\end{equation}
\end{lemma}

\begin{proof}
By replacing $V(\sigma)$ by $V(\sigma)+\epsilon \sigma^2$ if needed and taking a limit, we can assume that $V''$ is strictly positive at $\infty$.
Let $\chi$ be a smooth nonnegative function such that
$\chi(t)=1$ for $t\leq 1$ and $\chi(t)=0$ for $t\geq 2$.
For $\sigma\geq 0$, set
\begin{equation}
  V_R'''(\sigma) = V'''(\sigma)\chi(\sigma/R)
  .
\end{equation}
Define $V_R$ for $\sigma>0$ by
\begin{align}
  V_R''(\sigma) &= V''(0)+\int_0^\sigma V_R'''(\tau) \, d\tau,
  \\
  V_R'(\sigma) &= \int_0^\sigma V_R''(\tau) \, d\tau,
  \\
  V_R(\sigma) &= V(0) + \int_0^\sigma V_R'(\tau)\, d\tau,
\end{align}
and extend the definition to $\sigma\leq 0$ by $V_R(\sigma) = V_R(-\sigma)$.
Then $V_R$ agrees with $V$ on $[-R,R]$, is $C^3$ and $V_R''$ is bounded. 
Since $V''$ is increasing and strictly positive at $\infty$,
the same holds for $V_R''$ for sufficiently large $R$.
Finally, the limit~\eqref{eq_limit_VR} then follows from dominated convergence.  
\end{proof}

\subsection{DSS via a Polchinski-type equation}
\label{sec: DSS as a parabolic partial differential equation}

Our goal is to reformulate the DSS inequality \eqref{eq_prop_DSS} as $W_1 \geq 0$ for a suitable function $W_t$
with the condition \eqref{eq_prop_GHS} encoded as  $W_0 \geq 0$ and derive a parabolic partial differential equation for $W_t$. 
The DSS property will then be obtained using the maximum principle to show $W_t\geq 0$ for all times. 

A technical point is that this function $W$ will be defined on an unbounded domain with boundary.
At first reading one should think of the set $P \subset \Lambda$ below as equal to $\Lambda$.
The generalization to $P \subset \Lambda$ is needed for an induction argument, in Section~\ref{sec_induction}, 
to handle boundary values in the application of the maximum principle.
For a set $I\subset \Lambda$, we will use the notation
\begin{equation}
|\sigma|^2_I = \sum_{x\in I}\sigma_x^2,\qquad
\nabla^I=(\partial_x)_{x\in I},\qquad \Delta^I=\sum_{x\in I}\partial^2_x
\end{equation}
for the Euclidean norm, gradient, and Laplacian with respect to coordinates in that set. 

\begin{definition}
For $t \in (0,1]$ and $u\in\R^\Lambda$, for $P \subset \Lambda$, define the renormalised potential
\begin{equation}
v_t^{P,\Lambda}(u)
=
-\log \int_{\R^{\Lambda}} t^{-|P|/2}
\exp\bigg[ -\frac{|\sigma-u|^2_P}{2t} - \frac{|\sigma-u|^2_{\Lambda\setminus P}}{2}
-\frac{1}{2}(\sigma,A\sigma)-\sum_{x\in\Lambda} U_x(\sigma_x)\bigg] \, d\sigma
,
\label{eq_def_half_potential_DSS}
\end{equation}
where $U_x(\sigma_x) = V_x(\sigma_x) - \frac12 \sigma_x^2$,
and the corresponding fluctuation measure:
\begin{equation}
\langle\cdot\rangle_{u,t}^{P,\Lambda}
\propto
\int_{\R^{\Lambda}}t^{-|P|/2}
\exp\bigg[ -\frac{|\sigma-u|^2_P}{2t} - \frac{|\sigma-u|^2_{\Lambda\setminus P}}{2}
-\frac{1}{2}(\sigma,A\sigma)-\sum_{x\in\Lambda} U_x(\sigma_x)\bigg] (\cdot) \, d\sigma
.
\label{eq_def_product_measure_DSS}
\end{equation} 
Define also the renormalised potential for the doubled system:
\begin{equation}
V_t^{P,\Lambda}(u,v)
=
v_t^{P,\Lambda}(u+v) + v_t^{P,\Lambda}(u-v)
.
\label{eq_def_potential_DSS}
\end{equation}
For $t= 0$, the Gaussian measure is interpreted as degenerate and supported on $\{u_P\} \times \R^{\Lambda\setminus P}$.
\end{definition}

By symmetry of the GHS potential $V$ and thus of $U$, the renormalised potentials $v_t^{P,\Lambda}$ and $V_t^{P,\Lambda}$ are both symmetric for any $t \geq 0$.

Write $\id^P_{t}$ for the diagonal matrix with entries $t^{-1}{\bf 1}_P + {\bf 1}_{\Lambda\setminus P}$. 

\begin{lemma}\label{lemm_derivatives_V}
For each $u,v\in\R^{\Lambda}$ and $t \in (0,1]$,
\begin{align}
\nabla_1V^{P,\Lambda}_t(u,v)
&=
2\id^P_t u - \id^P_t\langle  \sigma\rangle_{u+v,t}^{P,\Lambda} - \id^P_t\langle  \sigma\rangle_{u-v,t}^{P,\Lambda},
\nnb 
\nabla_2V^{P,\Lambda}_t(u,v)
&=
2\id^P_t v - \id^P_t\langle  \sigma\rangle_{u+v,t}^{P,\Lambda} + \id^P_t\langle  \sigma\rangle_{u-v,t}^{P,\Lambda},
\label{eq_first_derivatives_DSS}
\\
\nabla^2_1V^{P,\Lambda}_t(u,v)
&=
2\id^P_t - \id^P_t \Big(\langle \sigma_x;\sigma_y\rangle_{u+v,t}^{P,\Lambda}\Big)_{x,y}\id^P_t
- \id^P_t \Big(\langle \sigma_x;\sigma_y\rangle_{u-v,t}^{P,\Lambda}\Big)_{x,y}\id^P_t 
,
\nnb 
\nabla^2_2V^{P,\Lambda}_t(u,v)
&=
2\id^P_t - \id^P_t \Big(\langle \sigma_x;\sigma_y\rangle_{u+v,t}^{P,\Lambda}\Big)_{x,y}\id^P_t
- \id^P_t \Big(\langle \sigma_x;\sigma_y\rangle_{u-v,t}^{P,\Lambda}\Big)_{x,y}\id^P_t 
,
\label{eq_second_derivatives_DSS}
\end{align}
where $\nabla_1$ denotes the gradient vector in $u$ and $\nabla_1^2$ the Hessian matrix in $u$
(and analogously for $v$),
and the notation $\avg{\cdot;\cdot}$ denotes the covariance with respect to $\avg{\cdot}$.
\end{lemma}

\begin{proof}
The first derivatives follow from
\begin{align}
\nabla_1 V_t^{P,\Lambda}(u,v) &= \nabla v_t^{P,\Lambda}(u+v) + \nabla v_t^{P,\Lambda}(u-v),
\\
\nabla_2 V_t^{P,\Lambda}(u,v) &= \nabla v_t^{P,\Lambda}(u+v) - \nabla v_t^{P,\Lambda}(u-v),
\end{align}
and
\begin{equation}
\nabla v_t^{P,\Lambda}(u) = \id_t^Pu - \id_t^P \avg{\sigma}^{P,\Lambda}_{u,t}.
\end{equation}
The second derivatives are similar.
\end{proof}

\begin{lemma} \label{lem_DSS_equiv}
The GHS condition~\eqref{eq_prop_GHS} is equivalent to:
\begin{equation}
\forall g\in\R^{\Lambda},\forall h\in\R_+^\Lambda,\forall x\in\Lambda,\qquad 
\nabla_2V^{\Lambda,\Lambda}_0(g,h)
\geq 
\nabla_2V^{\Lambda,\Lambda}_0(0,h)
.
\end{equation}
The DSS inequality~\eqref{eq_prop_DSS} is equivalent to:
\begin{equation}
\forall g\in\R^{\Lambda},\forall h\in\R_+^\Lambda,\forall x\in\Lambda,\qquad 
\nabla_2V^{P,\Lambda}_1(g,h)
\geq 
\nabla_2V^{P,\Lambda}_1(0,h)
.
\label{eq_HR_DSS_with_Vt}
\end{equation}
(For $t=1$ the function $V^{P,\Lambda}$ does not depend on $P$).
\end{lemma}

The first statement is not strictly necessary for the following as it will be generalised in Lemma~\ref{lem_DSS_initial_P},
but it illustrates the main idea.

\begin{proof}
  The claim for the DSS inequality is immediate from the identities of the previous lemma.
  For the GHS condition~\eqref{eq_prop_GHS}, note that the degenerate Gaussian measure in the definition of $v_t^{\Lambda,\Lambda}|_{t=0}$ is $\delta_u(d\sigma)$. Thus
  \begin{equation}
    v_0^{\Lambda,\Lambda}(u) = \frac12 (u,Au)+ \sum_{x\in\Lambda} U_x(u_x) + \text{constant}
  \end{equation}
  and
  \begin{equation}
    v_0^{\Lambda,\Lambda}(u\pm v) = \frac12 (u,Au) + \frac12 (v,Av) \pm (v,Au) +\sum_{x\in\Lambda} U_x(u_x \pm v_x) + \text{constant}
    .
  \end{equation}
  As a result,
  \begin{equation}
    V_0^{\Lambda,\Lambda}(u,v) = (u,Au) + (v,Av) + \sum_{x\in\Lambda} (U_x(u_x+v_x)+U_x(u_x-v_x)) + \text{constant}.
  \end{equation}
  Therefore
  \begin{equation}
    \nabla_{v_x} V_0^{\Lambda,\Lambda}(u,v) = 2(Av)_x + U_x'(u_x+v_x) -U_x'(u_x-v_x)
  \end{equation}
  and
  \begin{equation}
    \nabla_{v_x} V_0^{\Lambda,\Lambda}(u,v) - \nabla_{v_x} V_0^{\Lambda,\Lambda}(0,v)= U_x'(u_x+v_x) -U_x'(u_x-v_x)-\qb{U_x'(v_x)-U_x'(-v_x)}.
  \end{equation}
  Nonnegativity of the right-hand side is the GHS assumption~\eqref{eq_prop_GHS} on $V$.
  Finally, \eqref{eq_HR_DSS_with_Vt} follows directly from \eqref{eq_first_derivatives_DSS}.
\end{proof}


\begin{lemma} \label{lem_Polchinski_V}
For $t\in (0,1]$ and $(u,v)\in (\R^\Lambda)^2$,
$V_t(u,v)$ satisfies the following Polchinski-type equation:
\begin{equation}
2\partial_t V^{P,\Lambda}_t
=
\Delta^P_2V^{P,\Lambda}_t 
-\frac{1}{2}(\nabla^P_1 V^{P,\Lambda}_t)^2
-\frac{1}{2}(\nabla^P_2 V^{P,\Lambda}_t)^2
.
\label{eq_Polchinski_for_VP_DSS}
\end{equation}
\end{lemma}

\begin{proof}
The potential $v^{P,\Lambda}_t$ satisfies a simple instance of the Polchinski equation on $\R^\Lambda$:
\begin{equation}
\partial_t v^{P,\Lambda}_t
=
\frac{1}{2}\Delta^P v^{P,\Lambda}_t - \frac{1}{2}(\nabla^P v^{P,\Lambda}_t)^2
.
\end{equation}
For background on the Polchinski equation, see \cite{MR4798104}.
Here it is a simple computation that $z_t = e^{-v_t^{P,\Lambda}}$ satisfies the equivalent heat equation $\partial_t z_t = \frac12 \Delta^P z_t$.

Since $\nabla^2_1 V_t^{P,\Lambda} = \nabla^2_2 V_t^{P,\Lambda}$ as
immediate from the definition \eqref{eq_def_potential_DSS} of $V_t^{P,\Lambda}$ by symmetry,
\begin{align}
2 \partial_t V^{P,\Lambda}_t
&=
\frac{1}{2}\Delta^P_1V^{P,\Lambda}_t + \frac{1}{2}\Delta^P_2 V^{P,\Lambda}_t 
-\frac{1}{2}(\nabla^P_1 V^{P,\Lambda}_t)^2
-\frac{1}{2}(\nabla^P_2 V^{P,\Lambda}_t)^2
\nnb
&=
\Delta^P_2V^{P,\Lambda}_t 
-\frac{1}{2}(\nabla^P_1 V^{P,\Lambda}_t)^2
-\frac{1}{2}(\nabla^P_2 V^{P,\Lambda}_t)^2
.
\end{align}
The factor $2$ on the left-hand side is because differentiating in $u$ and $v$ each produce two terms.
\end{proof}

Let
\begin{equation}
W^{P,\Lambda}_t(u,v)
=
\nabla_2 V^{P,\Lambda}_t(u,v)-\nabla_2 V^{P,\Lambda}_t(0,v)
.
\end{equation}
We emphasise the crucial point that the Laplacian part of the Polchinski-type equation
\eqref{eq_Polchinski_for_VP_DSS} for $V_t^{P,\Lambda}(u,v)$ only involves derivatives in the $v$-variable.
As a consequence, as we show next, the vector-valued function $W^{P,\Lambda}_t$ satisfies a closed system of equations.

\begin{lemma} \label{lem_Polchinski_W}
Assume that $V$ is even. Then for $t \in (0,1]$ and $(u,v)\in (\R^\Lambda)^2$, for $x\in\Lambda$,
\begin{align}
2\partial_t W^{P,\Lambda}_t(u,v)_x
&=
\Delta^P_2 W^{P,\Lambda}_t(u,v)_x
-(\nabla_1 W^{P,\Lambda}_t(u,v)_x, \nabla_1 V^{P,\Lambda}_t(u,v))_P
\nnb
&\quad 
-(\nabla_2 W^{P,\Lambda}_t(u,v)_x, \nabla_2 V^{P,\Lambda}_t(u,v))_P
-(\nabla_2\nabla_{v_x} V^{P,\Lambda}_t(0,v), W^{P,\Lambda}_t(u,v))_P
,
\label{eq_PDE_W_DSS}
\end{align}
where $(u,v)_P = \sum_{x\in P} u_x v_x$.
%
\end{lemma}

\begin{proof}
Differentiating \eqref{eq_Polchinski_for_VP_DSS} gives
\begin{equation}
  2\partial_t \nabla_2 V_t^{P,\Lambda}
  = \Delta_2^P \nabla_2 V_t^{P,\Lambda}
  - (\nabla_1^P V_t^{P,\Lambda}, \nabla_1^P \nabla_2 V_t^{P,\Lambda})_P
  - (\nabla_2^P V_t^{P,\Lambda}, \nabla_2^P \nabla_2 V_t^{P,\Lambda})_P.
\end{equation}
For the first quadratic term, recall from \eqref{eq_def_potential_DSS} that 
\begin{equation}
\nabla_1 V_t^{P,\Lambda}(0,v) = \nabla v_t^{P,\Lambda}(v)+\nabla v_t^{P,\Lambda}(-v) =0
\end{equation}
since $v^{P,\Lambda}_t$ is even,
and hence
\begin{align}
   (\nabla_1^P V_t^{P,\Lambda}, \nabla_1^P \nabla_2 V_t^{P,\Lambda})_P(u,v)
& =(\nabla_1^P V_t^{P,\Lambda}, \nabla_1^P \nabla_2 V_t^{P,\Lambda})_P(u,v)
  -(\nabla_1^P V_t^{P,\Lambda}, \nabla_1^P \nabla_2 V_t^{P,\Lambda})_P(0,v)
  \nnb
  &=
  (\nabla_1^P V_t^{P,\Lambda}, \nabla_1^P W_t^{P,\Lambda})_P(u,v)
  .
\end{align}
For the second quadratic term, taking the difference using $a\cdot b - a'\cdot b' = (a-a')\cdot b' + a\cdot (b-b')$,
\begin{align}
   &(\nabla_2^P V_t^{P,\Lambda}, \nabla_2^P \nabla_2 V_t^{P,\Lambda})_P(u,v)
   -(\nabla_2^P V_t^{P,\Lambda}, \nabla_2^P \nabla_2 V_t^{P,\Lambda})_P(0,v)
   \nnb
   &= 
   (W_t^{P,\Lambda}(u,v), \nabla_2^P \nabla_2 V_t^{P,\Lambda}(0,v))_P
   +(\nabla_2^P V_t^{P,\Lambda}(u,v), \nabla_2^P W_t^{P,\Lambda}(u,v))_P
   .
\end{align}
Combining gives the desired equation.
\end{proof}

The next lemma shows that the parabolic system of equations \eqref{eq_PDE_W_DSS} for $W_t^{P,\Lambda}$ is cooperative.

\begin{lemma} \label{lem_V_FKG}
$\nabla_{v_x}\nabla_{v_y} V_t^{P,\Lambda}(u,v)\leq 0$ for all $x \neq y$.
\end{lemma}

\begin{proof}
Follows from the FKG inequality, see \eqref{eq_second_derivatives_DSS}.
\end{proof}

\begin{lemma} \label{lem_V_growth}
Assume $V''$ is bounded and strictly positive at $\infty$. Then $\nabla^2 v_t^{P,\Lambda}(u)$ is bounded uniformly in $t\in [0,1]$ and $u\in\R^\Lambda$.
\end{lemma}

\begin{proof}
%
To get the uniformity in $t\in [0,1]$, expressions of the form \eqref{eq_second_derivatives_DSS} are not convenient. Instead, we first  change variables in \eqref{eq_def_half_potential_DSS}:
\begin{equation}
v_t^{P,\Lambda}(u)
=
-\log \int_{\R^{\Lambda}} t^{-|P|/2}
\exp\bigg[ -\frac{|\sigma|^2_P}{2t} - \frac{|\sigma|^2_{\Lambda\setminus P}}{2}
-\frac{1}{2} \big( \sigma + u ,A (\sigma+u) \big)-\sum_{x\in\Lambda} U_x(\sigma_x + u_x)\bigg] \, d\sigma,
\end{equation}
and then take the derivatives:
\begin{align}
\nabla_{u_x}\nabla_{u_y} v^{P,\Lambda}_t(u)
&= A_{xy} + \delta_{xy}\E\big[U_x''(\sigma_x+u_x)\big]
\nnb
&\qquad - \cov\big((A\sigma)_x+U_x'(\sigma_x+u_x), (A\sigma)_y+ U_y'(\sigma_y+u_y)\big),
\end{align}
where expectation and covariance are under the measure $\avg{\cdot}_{u,t}^{P,\Lambda}$.
Since $\frac12(\sigma,A\sigma)+\sum V_x(\sigma)$ can be written as 
a bounded perturbation of a uniformly convex potential (with the bound depending on $\Lambda$),
the measure satisfies a log-Sobolev inequality (and in particular a spectral gap inequality)
uniformly in  $t\in [0,1]$ and in $u\in \R^\Lambda$ (with the constant depending on $\Lambda$).
Thus
\begin{equation}
\var(U_x'(\sigma_x+u_x)) \lesssim \E[U_x''(\sigma_x+u_x)^2]
\lesssim 1,
\qquad
\var((A\sigma)_x) \lesssim 1,
\end{equation}
where $\lesssim$ denotes inequality up to a constant independent of $t\in[0,1]$ and the field $u$.
The expectations are also bounded since each 
$U''_x = V''_x -1$ is bounded (see Section \ref{sec: Regularisation}).
\end{proof}

\subsection{Maximum principle}
Let
\begin{equation}
X_P = \R^\Lambda\times (\R^P_+\times\{0_{\Lambda\setminus P}\})
\end{equation}
and define the (degenerate) parabolic boundary
\begin{equation}
\partial(\R_+\times X_P)
= (\R_+\times \partial X_P) \cup (\{t=0\} \times X_P)
\end{equation}
where
\begin{equation}
  \partial X_P = \R^\Lambda \times (\partial \R_+^P \times \{0_{\Lambda\setminus P}\}).
\end{equation}

\begin{proposition}\label{prop_MP_DSS}
Suppose $W^{P,\Lambda}\geq 0$ entrywise on the (degenerate) parabolic boundary $\partial(\R_+\times X_P)$.
Then $W^{P,\Lambda} \geq 0$ entrywise on $\R_+\times X_P$.
\end{proposition}

\begin{proof}
The vanishing components in $\Lambda\setminus P$ do not play a role and the statement is really about $\R^\Lambda$-valued functions on $\R^\Lambda \times \R_+^{P}$.
The equation \eqref{eq_PDE_W_DSS} for $W^{P,\Lambda}_t$ is a system of linear equations
\begin{equation}
  \partial_t U_i = (\nabla^T A \nabla)U_i + \sum_j B_{j}\nabla_j U_i + \sum_j C_{ij} U_j = L[U_i] + \sum_j C_{ij} U_j,
\end{equation}
where $U$ is an $\R^N$-valued function defined on an unbounded domain $\Omega = \R^N \times \R_+^M$,
$A$ is a constant positive semidefinite matrix, $B$ a time-dependent vector-valued function,
and $C$ a time-dependent matrix-valued function with $C_{ij} \geq 0$ for $i \neq j$,
by Lemma~\ref{lem_V_FKG}.

The maximum principle for cooperative parabolic systems applies under growth assumptions on the solution $U$ and the coefficients $B,C$.
It follows from \cite[Theorem~1]{MR374632} that the maximum principle holds provided that $B$ is linearly bounded, $C$ is quadratically bounded,
and $U$ has at most Gaussian exponential growth.
Here, Lemma~\ref{lem_V_growth} guarantees that $B$ has linearly growth, $C$ is actually bounded, and $U$ has linear growth, so all assumptions
are easily satisfied.
An alternative (and more general) reference would be \cite[Theorem~2]{MR612130} and the discussion following it.
\end{proof}

\subsection{Induction and proof of Theorem~\ref{thm:main}}\label{sec_induction}

The maximum principle of Proposition~\ref{prop_MP_DSS} with $P=\Lambda$ would imply the DSS inequality, assuming it holds on the boundary.
The boundary assumption is removed by an induction over the size of the boundary, corresponding to a double  induction on the size of the sets $P\subset\Lambda$. 
Let DSS($P,\Lambda$) be the assumption that, for all ferromagnetic matrix $A\in\R^{\Lambda\times\Lambda}$ and all GHS potentials $V$ as in Theorem~\ref{thm:main}:
\begin{equation}
W^{P,\Lambda}\geq 0\quad\text{entrywise on } \R_+\times X_P
.
\tag{$\text{DSS}(P,\Lambda)$}
\end{equation}
To prove the full DSS inequality \DSS$(\Lambda,\Lambda)$, we apply Proposition~\ref{prop_MP_DSS} inductively with 
assumption \DSS($P,\Lambda$). 
This idea is similar to the original approach of Ding--Song--Sun \cite{MR4586225}, who also induct over the size of the set on which $h$ vanishes.
The next two lemmas carry out the induction step.

\begin{lemma}[Space boundary] \label{lem_DSS_boundary}
Let $P \subseteq \Lambda$, and assume \DSS$(P',\Lambda)$ for all $P' \subsetneq P$. Then, for all $t \in (0,1]$,
\begin{equation}
W^{P,\Lambda}_t(g,h)\geq 0\quad\text{entrywise for }(g,h)\in \partial  X_P.
\end{equation}
\end{lemma}

\begin{proof}
If $(g,h)\in\partial X_P$, then there is a set $P'\subsetneq P$ such that $h_x=0$ if $x\notin P'$. 
The claim follows from the inductive assumption and a version of Lemma~\ref{lem_DSS_equiv}.
More precisely,
\begin{align}
&\Big(\langle
\sigma\rangle^{P,\Lambda}_{h,t} 
-\langle
\sigma\rangle^{P,\Lambda}_{-h,t}\Big)
-\Big(\langle
\sigma\rangle^{P,\Lambda}_{g+h,t} 
-\langle
\sigma\rangle^{P,\Lambda}_{g-h,t}\Big)
= 
\Big(\widetilde{\langle
\sigma\rangle}^{P',\Lambda}_{h,t} 
-\widetilde{\langle
\sigma\rangle}^{P',\Lambda}_{-h,t}\Big)
-\Big(\widetilde{\langle
\sigma\rangle}^{P',\Lambda}_{\tilde g+h,t} 
-\widetilde{\langle
\sigma\rangle}^{P',\Lambda}_{\tilde g-h,t}\Big)
,
\end{align}
where the measure on the right-hand side has modified couplings
\begin{equation}
  \tilde A = A + \Big(\frac{1}{t}-1\Big)1_{P\setminus P'},
  \qquad
  \tilde g_x = g_x + \Big(\frac{1}{t}-1\Big)1_{x\in P \setminus P'}g_x,
\end{equation}
and \DSS$(P',\Lambda)$ can be applied to obtain the conclusion.
\end{proof}

\begin{lemma}[Initial condition] \label{lem_DSS_initial_P}
Assume \DSS$(\Lambda',\Lambda')$ for all $\Lambda' \subsetneq \Lambda$. Then, for $P \neq \varnothing$,
\begin{equation}
W^{P,\Lambda}_0(g,h)\geq 0\quad\text{entrywise for } (g,h) \in \R^\Lambda \times \R_+^\Lambda
.
\end{equation}
If $|\Lambda|=1$ then the conclusion holds only under the GHS condition \eqref{eq_prop_GHS} on $V$.
\end{lemma}

\begin{proof}
From the definition~\eqref{eq_def_potential_DSS} of $V^{P,\Lambda}_t$, 
interpreted as a degenerate Gaussian if $t=0$:
\begin{align}
v_0^{P,\Lambda}(u)
&=
\sum_{x\in P} U_x(u_x)
\nnb
&\quad-
\log \int_{\R^{P^c}}\exp\bigg[  
- \frac{|\sigma-u|^2_{\Lambda\setminus P}}{2}
-\frac{1}{2}\big((\sigma_{P^c},u_P),A(\sigma_{P^c} ,u_P)\big)
-\sum_{x\in\Lambda\setminus P} U_x(\sigma_x)\bigg] \, d\sigma
.
\end{align}
For subsets $I,J\subset\Lambda$, write $A_{I,J}$ for the matrix $(A_{x,y})_{x\in I,y\in J}$. 
Expanding the quadratic forms and rearranging, 
we find that the integral term above is given by:
\begin{align}
&
-\frac{\big|A_{P^c,P}u_P\big|^2}{2} 
+(u_{P^c},A_{P^c,P}u_P)
+\frac{(u_P,A_{P,P}u_P)}{2} 
\nnb
&\qquad\qquad
-\log \int_{\R^{P^c}}\exp\bigg[  - \frac{|\sigma-u+A_{P^c,P}u_P|^2_{P^c}}{2}
-\frac{1}{2}\big(\sigma,A_{P^c,P^c}\, \sigma)\big) -\sum_{x\in P^c} U_x(\sigma_x)\bigg] \, d\sigma
.
\end{align}
The last integral is the partition function of the measure $\langle\cdot\rangle^{P^c,P^c}_{\tilde u,t=1}= \langle\cdot\rangle^{P^c}_{\tilde u,t=1}$ (recall its definition~\eqref{eq_def_product_measure_DSS}), 
with
\begin{equation}
\tilde u = u_{P^c}-A_{P^c,P}\, u_P = \tilde A u.
\end{equation}
Since $A$ is ferromagnetic, the $P^c \times \Lambda$ matrix $\tilde A$ has nonnegative entries.
Thus
\begin{equation}
  \nabla v_0^{P,\Lambda}(u)
  = (1_{x\in P} U_x'(u_x))_x - \tilde A^T \avg{\sigma}_{\tilde A u,1}^{P^c} + \text{(affine in $u$)}
\end{equation}
and therefore
\begin{align}
  \nabla_{v} V_0^{P,\Lambda}(u,v)
  &= (1_{x\in P} \pb{U_x'(u_x+v_x)-U_x'(u_x-v_x)})_x
  \nnb
  &\qquad
  - \tilde A^T \avg{\sigma}_{\tilde A(u+v),1}^{P^c}
  + \tilde A^T \avg{\sigma}_{\tilde A(u-v),1}^{P^c} + \text{(linear in $v$)}
\end{align}
and
\begin{align}
  W_0^{P,\Lambda}(g,h)
  &= \Big(1_{x\in P} \pb{U_x'(g_x+h_x)-U_x'(g_x-h_x)-(U_x'(h_x)-U_x'(-h_x))}\Big)_x
  \nnb
  &\qquad
  - \tilde A^T \pb{
  \avg{\sigma}_{\tilde A(g+h),1}^{P^c}
  -\avg{\sigma}_{\tilde A(g-h),1}^{P^c}
  -\pb{\avg{\sigma}_{\tilde Ah,1}^{P^c}
   -\avg{\sigma}_{-\tilde Ah,1}^{P^c}}
  }
  .
\end{align}
The GHS assumption \eqref{eq_prop_GHS} implies that the $U$ terms are nonnegative. 
The assumption that \DSS$(P^c,P^c)$ holds for $P^c \subsetneq \Lambda$ together with $\tilde A \geq 0$ entrywise shows that the second line is also nonnegative.
\end{proof}

\begin{proof}[Conclusion of the proof]
  The goal is to prove \DSS$(\Lambda,\Lambda)$. 
  By induction, assume that for all $\varnothing\neq \Lambda' \subsetneq \Lambda$ one has \DSS($\Lambda',\Lambda')$
  and that for all $P \subsetneq \Lambda$ one has \DSS$(P,\Lambda)$. 
  Then Proposition~\ref{prop_MP_DSS} together with Lemmas~\ref{lem_DSS_boundary}--\ref{lem_DSS_initial_P} imply \DSS($P,\Lambda$). 
  It therefore only remains to check the initial step, corresponding to proving \DSS$(\varnothing,\Lambda)$ for each $\Lambda$. 
  Notice however that \DSS$(\varnothing,\Lambda)$ is trivially true since $W^{\varnothing,\Lambda}_t(0,g)=0$ by Lemma~\ref{lemm_derivatives_V}, 
  for any $g\in\R^\Lambda$ and $t\geq 0$. 
\end{proof}


\section*{Acknowledgements}

We thank Nestor Guillen, Michael Hofstetter, Chuck Newman, and Ofer Zeitouni for various related discussions.

This work was supported in part by NSF grant DMS-2348045, the Simons Collaboration
grant on Probabilistic Paths to Quantum Field Theory,
and a grant of access to OpenAI models through the ChatGPT for Academic Researchers program.

\section*{AI statement}

The main idea for the proof, in particular the combination of Ginibre-type doubling and the maximum principle
as well as the use of the key quantity \eqref{eq_prop_DSS} identified in the paper \cite{MR4586225} instead of the two-point function directly,
 was found without any AI assistance. In the writing process, AI was used for checking, finding references
 for the maximum principle under sufficiently general conditions, but no other arguments.
In May 2026 and earlier, Gemini and ChatGPT 5.4 and 5.5 were used in preliminary brainstorming of the approach of using the maximum principle to prove various 
 correlation inequalities, including the FKG and GHS inequalities. For the GHS inequality, GPT-5.5 provided input on how to use
 induction to handle the boundary condition on the positive field.

\bibliography{all}
\bibliographystyle{plain}

\end{document}